\documentclass[12pt,reqno]{amsart}
\usepackage{amssymb}
\usepackage{mathrsfs}
\usepackage{palatino}
\usepackage{bbm}

\newcommand{\CC}{\mathbb{C}}

\newcommand{\ZZ}{\mathbb{Z}}

\newcommand{\gl}{\mathfrak{gl}}
\newcommand{\fs}{\mathfrak{s}}

\newcommand{\Ymn}{Y_{M|N}}

\newcommand{\glmn}{\mathfrak{gl}_{M|N}}

\newcommand{\mfs}{\check{\mathfrak{s}}}
\newcommand{\sd}{\mathfrak{s}^\dagger}
\newcommand{\sr}{\mathfrak{s}^r}

\DeclareMathOperator{\sgn}{sgn}
\DeclareMathOperator{\Mat}{Mat}

\DeclareMathOperator{\qdet}{qdet}
\DeclareMathOperator{\qsdet}{qsdet}

\usepackage{mathrsfs}
\usepackage{cases}
\usepackage{epic}
\usepackage{amsfonts}
\usepackage{graphicx}
\usepackage{amsmath}
\usepackage{amssymb, upgreek}
\usepackage{bm}
\usepackage{latexsym,todonotes}
\usepackage{pdflscape}
\usepackage[all]{xypic}
\usepackage[all]{xy}
\usepackage{color}
\usepackage{colordvi}
\usepackage{multicol}
\usepackage[normalem]{ulem}
\usepackage[linktocpage=true]{hyperref}
\usepackage{hyperref}
\hypersetup{colorlinks,linkcolor=blue,urlcolor=cyan,citecolor=red}
\numberwithin{equation}{section}
\newtheorem{Theorem}{Theorem}[section]
\newtheorem{Lemma}[Theorem]{Lemma}
\newtheorem{Corollary}[Theorem]{Corollary}
\newtheorem{Proposition}[Theorem]{Proposition}

\theoremstyle{Theorem}

\newtheorem*{thm*}{Theorem}
\newtheorem*{thm**}{Corollary}
\newtheorem*{thm***}{Theorem B}
\theoremstyle{remark}
\newtheorem{Remark}{Remark}

\newtheorem*{Definition}{Definition}
\newtheorem*{Example}{Example}

\numberwithin{equation}{section}
\begin{document}
\title{A note on the center of the super Yangian $\Ymn(\fs)$}
\author[Hao Chang \lowercase{and} Hongmei Hu]{Hao Chang \lowercase{and} Hongmei Hu*}
\address[H. Chang]{School of Mathematics and Statistics, and Key Laboratory of Nonlinear Analysis \& Applications (Ministry of Education), Central China Normal University, Wuhan 430079, People's Republic of China}
\email{chang@ccnu.edu.cn}
\address[H. Hu]{School of Mathematical Sciences, Suzhou University of Science and Technology, 215009 Suzhou, People's Republic of China}
\email{hmhu@usts.edu.cn}
\thanks{* Corresponding author.}

\makeatletter
\makeatletter
\@namedef{subjclassname@2020}{%
  \textup{2020} Mathematics Subject Classification}
\makeatother
\makeatother
\begin{abstract}
Let $\Ymn(\fs)$ be the super Yangian associated with an arbitrary fixed $0^M1^N$-sequence $\fs$.
In the present paper,
we give a new formula for the quantum Berezinian by using the parabolic generators,
which generalizes the usual expression in terms of RTT generators or Drinfeld generators.
\end{abstract}
\maketitle
\setcounter{tocdepth}{2}
\section{Introduction}
The Yangians were introduced  by Drinfeld in his fundamental paper \cite{D85},
they from a remarkable family of quantum groups related to rational solutions of the classical {\it Yang-Baxter equation}.
The Yangian $Y_N$ associated to the Lie algebra $\mathfrak{gl}_N$ was earlier considered in the work of Faddeev and the St. Petersburg school
around 1980.
It is an associative algebra, which can be defined using the {\it RTT formalism} of Faddeev, Reshetikhin and Takhtadzhyan \cite{FRT90}.

In \cite{BK05}, Brundan and Kleshchev found a {\it parabolic presentation} for $Y_N$ accociated to each composition $\mu$ of $N$.
The new presentation corresponds to a block matrix decompositionof $\mathfrak{gl}_N$ of shape $\mu$.
In the special case when $\mu=(N)$,
this presentation is exactly the original {\it RTT presentation}.
On the other extreme case when $\mu=(1,\dots,1)=(1^N)$,
the corresponding parabolic presentation is just a variation of Drinfeld's from \cite{D88};
see \cite[Remark 5.12]{BK05}.

The super Yangian associated with the Lie superalgebra $\mathfrak{gl}_{M|N}$ was defined by Nazarov \cite{Na91} in terms of the RTT presentation as a super analogue of $Y_N$.
A Drinfeld-type presentation corresponding to a standard {\it $01$-sequence} was obtained by Gow \cite{Gow07}.
The parabolic presentations of the super Yangian $Y_{M|N}(\fs)$ associated with arbitrary $01$-sequence $\fs$ were given by Peng \cite{Peng16}.
Later, Tsymbaliuk \cite{Tsy20} introduced the Drinfeld super Yangian, which recovered the construction of \cite{Peng16} in a particular case.
Recently, the finite-dimensional irreducible representations of the super Yangian $Y_{M|N}(\fs)$ are described with the use of the {\it odd reflections} (\cite{Mol22,Lu22}).

It is well-known that the  the centre $Z(Y_N)$ of $Y_N$ is generated by the coefficients of the {\it quantum
determinant} $\qdet T(u)$, see \cite[Theorem 2.13]{MNO96}.
Moreover, the quantum determinant $\qdet T(u)$ admits the factorization in terms of the diagonal Drinfeld-type generators (see \cite[Theorem 1.10.5]{Mol07}, \cite[Theorem 8.6]{BK05}).
For the super Yangian $\Ymn$, Nazarov defined the {\it quantum Berezinian} $\qsdet T(u)$ which plays a similar role in the study of the super Yangian $Y_{M|N}$ as the quantum determinant does in the case of $Y_N$.
The quantum Berezinian $\qsdet T(u)$ can also be written as a formula in terms of the diagonal Drinfeld generators (\cite[Theorem 1]{Gow05}).
In his article \cite{Tsy20}, Tsymbaliuk also gave a description of the center of the super Yangian  associated with arbitrary $01$-sequence by using the quantum Berezinian which is defined via the diagonal Drinfeld generators (\cite[(2.41)]{Tsy20}, see also \cite[Remark 2.5]{BG19}), thus this generalizes the results of \cite{Na91},\cite{Gow05}.

Parabolic presentations play a central role in the study of the finite $W$-algebras and finite $W$-superalgebras (see \cite{BK06, Peng21}).
Our goal in this paper is to give a new formula for the quantum Berezinian in terms of the diagonal parabolic generators.
We organize this article in the following manner.
In Section \ref{section:basic},
we recall some basic properties of the (super) Yangian $Y_{M|N}$ and the parabolic generators.
For the Yangian $Y_N$,
it is shown in Section \ref{section:yangian} that the quantum determinant $\qdet T(u)$ admits a factorization in terms of the diagonal parabolic generators.
In Section \ref{section:superyangian}, we define the quantum Berezinian $\qsdet T(u)$ in terms of the RTT generators for arbitrary $01$-sequence.
This definition in the case of standard $01$-sequence goes back to \cite{Na91}.
We will show that the quantum Berezinian can be decomposed into the product in terms of the diagonal parabolic generators.
In particular, for the tuple $(1,\dots,1)=(1^{M+N})$,
this decomposition coincides with Tsymbaliuk's formula from \cite{Tsy20}.

\section{The super Yangian $\Ymn$}\label{section:basic}
\subsection{RTT Presentation}\label{subsection RTT presentation}
Let $\fs$ be a $0^M1^N$-sequence (or $01$-sequence for short) of $\glmn$,
which is a sequence consisting of $M$ $0'$s and $N$ $1'$s,
arranged in a row with respect to a certain order.
It is well-known \cite[Section 1.3]{CW13} that there is a bijection between the set of $0^M1^N$-sequence and the Weyl group orbits of simple systems of $\glmn$.

For a given $\fs$,
the {\it super Yangian} associated to the general linear Lie superalgebra $\glmn$ is a unital associative superalgebra over $\CC$ generated by the {\it RTT generators} \cite{Na91}
$\{t_{i,j}^{(r)};~i,j=1,\dots,M+N,~r\geq 1\}$ subject to the following relations:
\begin{align}\label{RTT relations}
\left[t_{i,j}^{(r)}, t_{k,l}^{(s)}\right] =(-1)^{|i||j|+|i||k|+|j||k|}\sum_{t=0}^{\min(r,s)-1}
\left(t_{k, j}^{(t)} t_{i,l}^{(r+s-1-t)}-
t_{k,j}^{(r+s-1-t)}t_{i,l}^{(t)}\right),
\end{align}
where $|i|$ denotes the $i$-th digit of $\fs$ and the parity of $t_{i,j}^{(r)}$ for $r>0$ is defined by $|i|+|j|~(\text{mod}~2)$,
and the bracket in \eqref{RTT relations} is understood as the supercommutator.
The algebra is denoted by $\Ymn(\fs)$ (or $\Ymn$ when no confusion may occur).
The element $t_{i,j}^{(r)}$ is called an {\it even} ({\it odd}, respectively) element if its parity is $0$ ($1$, respectively).

The original definition in \cite{Na91} corresponds to the case when $\fs$ is the standard one, that is
\begin{equation}\label{standard sequence}
\fs^{st}=\stackrel{M}{\overbrace{0\ldots0}}\,\stackrel{N}{\overbrace{1\ldots1}}.
\end{equation}

\begin{Remark}\label{rk1}
When $N=0$ and $\fs=\fs^{st}$, the super Yangian $\Ymn$ is natural isomorphic to the usual Yangian $Y_M$.
\end{Remark}

By convention, we set $t_{i,j}^{(0)}:=\delta_{i,j}$.
We often put the generators $t_{i,j}^{(r)}$ for all $r\geq 0$ to form the power series
\begin{align*}
t_{i,j}(u):= \sum_{r \geq 0}t_{i,j}^{(r)}u^{-r} \in\Ymn[[u^{-1}]].
\end{align*}
These power series for all $1\leq i,j\leq M+N$ can be collected together into a single matrix
\[T(u):=\big(t_{i,j}(u)\big)_{1\leq i,j\leq M+N}\in\Mat_{M+N}(\Ymn[[u^{-1}]]).\]
Note that the matrix $T(u)$ is invertible,
we observe the following notation for the entries of the inverse of the matrix $T(u)$:
$$T(u)^{-1}=:\left(t_{i,j}'(u)\right)_{i,j=1}^{M+N}.$$
In terms of generating series, defining relation \eqref{RTT relations} are equivalent to
\begin{align}\label{tiju tklu relation}
[t_{i,j}(u),t_{k,l}(v)]=\frac{(-1)^{|i||j|+|i||k|+|j||k|}}{(u-v)}(t_{k,j}(u)t_{i,l}(v)-t_{k,j}(v)t_{i,l}(v)).
\end{align}

\subsection{Maps between super Yangians}
There are some notations related to any fixed $0^M1^N$-sequence $\fs$ in \cite[Section 4]{Peng16}:
\begin{enumerate}
\item $\mfs$:= the $0^N1^M$-sequence obtained by interchanging the 0's and 1's of $\fs$.
\item $\sr$:= the reverse of $\fs$.
\item $\sd$:= $(\mfs)^r$, the reverse of $\mfs$.
\end{enumerate} If $\fs_1$ and $\fs_2$ are two $01$-sequences, then $\fs_1\fs_2$ simply means the concatenation of $\fs_1$ and $\fs_2$.

For future reference, we recall
the following result, cf. \cite[Proposition 4.1]{Peng16}; see also \cite[Section 4]{Gow07}.

\begin{Proposition}\label{morphisms}
	\begin{enumerate}
		\item The map $\rho_{M|N}:\Ymn(\fs)\rightarrow Y_{N|M}(\sd)$ defined by
		\[
		\rho_{M|N}\big(t_{ij}(u)\big)=t_{M+N+1-i,M+N+1-j}(-u)
		\]
		is an isomorphism.
		\item The map $\omega_{M|N}:\Ymn(\fs)\rightarrow \Ymn(\fs)$ defined by
		\[\omega_{M|N}\big(T(u)\big)=\big(T(-u)\big)^{-1}  \]
		is an automorphism.
		\item The map $\zeta_{M|N}:\Ymn(\fs)\rightarrow Y_{N|M}(\sd)$ defined by
		\[\zeta_{M|N}=\rho_{M|N}\circ\omega_{M|N}\]
		is an isomorphism.
		\item Let $p,q\in\ZZ_{\geq 0}$ and let $\fs_1$ be an arbitrary $0^p1^q$-sequence. Let
		\[\varphi_{p|q}:\Ymn(\fs)\rightarrow Y_{p+M|q+N}(\fs_1\fs)\]
		be the injective algebra homomorphism sending each $t_{i,j}^{(r)}$ in $\Ymn(\fs)$ to $t_{p+q+i,p+q+j}^{(r)}$ in $Y_{p+M|q+N}(\fs_1\fs)$.
		Then the map $\psi_{p|q}:\Ymn(\fs)\rightarrow Y_{p+M|q+N}(\fs_1\fs)$ defined by
		\[\psi_{p|q}=\omega_{p+M|q+N}\circ\varphi_{p|q}\circ\omega_{M|N},\]
		is an injective homorphism.
\end{enumerate}
\end{Proposition}
We call $\psi_{p|q}$ the {\it shift map} and $\zeta_{M|N}$ the {\it swap map}.

\subsection{Parabolic generators}\label{subsection:parabolic generators}
Let $\mu=(\mu_1,\ldots,\mu_n)$ be a given composition of $M+N$ with length $n$ and
fix a $0^M1^N$-sequence $\fs$.

By definition, the leading minors of the matrix $T(u)$ are invertible.
Then it possesses a {\it Gauss decomposition} (cf. \cite{GR97}) with respect to $\mu$
\[
T(u)=F(u)D(u)E(u)
\]
for unique {\it block matrices}
\[
D(u) = \left(
\begin{array}{cccc}
D_{1}(u) & 0&\cdots&0\\
0 & D_{2}(u) &\cdots&0\\
\vdots&\vdots&\ddots&\vdots\\
0&0 &\cdots&D_n(u)
\end{array}
\right),
$$$$
E(u) =
\left(
\begin{array}{cccc}
I_{\mu_1}& E_{1,2}(u) &\cdots&E_{1,n}(u)\\
0 & I_{\mu_2}&\cdots&E_{2,n}(u)\\
\vdots&\vdots&\ddots&\vdots\\
0&0 &\cdots&I_{\mu_n}
\end{array}
\right),\:
F(u) = \left(
\begin{array}{cccc}
I_{\mu_1}& 0 &\cdots&0\\
F_{2,1}(u) & I_{\mu_2} &\cdots&0\\
\vdots&\vdots&\ddots&\vdots\\
F_{n,1}(u)&F_{n,2}(u) &\cdots&I_{\mu_n}
\end{array}
\right),
\]
where
\begin{align}\label{definition of Da}
D_a(u)=(D_{a;i,j}(u))_{1\leq i,j\leq\mu_a}
\end{align}
\begin{align}\label{definition of Eab}
E_{a,b}(u)=(E_{a,b;i,j}(u))_{1\leq i\leq\mu_a,1\leq j\leq\mu_b},\quad
F_{b,a}(u)=(F_{b,a;i,j}(u))_{1\leq i\leq\mu_b,1\leq j\leq\mu_a}
\end{align}
are $\mu_a\times\mu_a,\mu_a\times\mu_b$ and $\mu_b\times\mu_a$ matrices, respectively,
for all $1\leq a\leq n$ in (\ref{definition of Da}) and all $1\leq a<b\leq n$ in (\ref{definition of Eab}).
The entries of these matrices define power series
\[D_{a;i,j}(u):=\sum_{r\geq 0}D_{a;i,j}^{(r)}u^{-r},\quad E_{a,b;i,j}(u):=\sum_{r\geq 1}E_{a,b;i,j}^{(r)}u^{-r},\quad
F_{b,a;i,j}(u):=\sum_{r\geq 1}F_{b,a;i,j}^{(r)}u^{-r}.\]
Also note that all the submatrices $D_{a}(u)$'s are invertible, and we define the $\mu_a\times\mu_a$ matrix
$D_a^{\prime}(u)=\big(D_{a;i,j}^{\prime}(u)\big)_{1\leq i,j\leq \mu_a}$ by
$D_a^{\prime}(u):=\big(D_a(u)\big)^{-1}$.

In fact, the algebra $\Ymn$ is generated by the elements $\{D_{a;i,j}^{(r)};~1\leq a\leq n, 1\leq i,j\leq \mu_a, r\geq 0\}$,
$\{E_{a;i,j}^{(r)};~1\leq a<n, 1\leq i\leq \mu_a, 1\leq j\leq \mu_{a+1},r\geq 1\}$ and $\{F_{a;i,j}^{(r)};~1\leq a<n, 1\leq i\leq \mu_{a+1}, 1\leq j\leq \mu_{a},r\geq 1\}$ (\cite[Theorem 3.4 and Remark 3.5]{Peng16}).
These generators are called {\it parabolic generators}.
In the following,
we will use the notation $\Ymn(\fs)$ or $\Ymn(\mu)$ or $Y_{\mu}(\fs)$ to emphasize the choice of $\fs$ or $\mu$ or both when necessary.

When necessary,
we will add an additional superscript $\mu$ to our notation to avoid any ambiguity as $\mu$ varies.
We write the matrix $T(u)$ in block form as
\[
T(u) =
\left(
\begin{array}{ccc}
^{\mu}T_{1,1}(u) &\cdots&^{\mu}T_{1,n}(u)\\
\vdots&\ddots&\vdots\\
^{\mu}T_{n,1}(u) &\cdots&^{\mu}T_{n,n}(u)
\end{array}
\right),
\]
where $^{\mu}T_{a,b}(u)$ is a $\mu_a\times\mu_b$ matrix.
In terms of quasideterminants of \cite{GR97},
we have the following description (cf. \cite[(3.6)]{Peng16}):
\begin{align}\label{quasideterminant D}
D_a(u) =
\left|
\begin{array}{cccc}
^{\mu}T_{1,1}(u) & \cdots &^{\mu}T_{1,a-1}(u)&^{\mu}T_{1,a}(u)\\
\vdots & \ddots &\vdots&\vdots\\
^{\mu}T_{a-1,1}(u)&\cdots&^{\mu}T_{a-1,a-1}(u)&^{\mu}T_{a-1,a}(u)\\
^{\mu}T_{a,1}(u) & \cdots & ^{\mu}T_{a,a-1}(u)&
\hbox{\begin{tabular}{|c|}\hline$^{\mu}T_{a,a}(u)$\\\hline\end{tabular}}
\end{array}
\right|.
\end{align}

In view of \cite[(4.1)]{Peng16},
we have
\begin{equation}\label{psit}
  \psi_{p|q}\big(t_{ij}(u)\big)=
  \left| \begin{array}{cccc} t_{1,1}(u) &\cdots &t_{1,p+q}(u) &t_{1, p+q+j}(u)\\
         \vdots &\ddots &\vdots &\vdots \\
         t_{p+q,1}(u) &\cdots &t_{p+q,p+q}(u) &t_{p+q, p+q+j}(u)\\
         t_{p+q+i, 1}(u) &\cdots &t_{p+q+i,p+q}(u) &\boxed{t_{p+q+i, p+q+j}(u)}
         \end{array} \right|.
         \end{equation}

Note that \eqref{psit} implies that $\psi_{p|q}$ depends only on $p+q$ so we may simply write $\psi_{p|q}=\psi_{p+q}$ when appropriate.
Then (\ref{quasideterminant D}) and \eqref{psit} would imply that
\begin{align}\label{psi D1=Da}
D_{a;i,j}(u)=\psi_{\mu_1+\cdots+\mu_{a-1}}(D_{1;i,j}(u))
\end{align}

In the special case when $\mu=(M+N)$,
we have $T(u)= \!^{\mu}T_{1,1}(u)$.
The corresponding parabolic generators is exactly the original RTT generators $\{t_{i,j}^{(r)};~1\leq i,j\leq M+N; r\geq 1\}$ in subsection \ref{subsection RTT presentation}.
At another extreme, for $\mu=(1,\cdots, 1)$,
the parabolic generators is called the {\it Drinfeld generators}.
In this case, all $D_a(u)$'s are $1\times 1$ matrices.
We will use $d_a(u)$ instead of $D_a(u)$ for conciseness.

\section{The center of $Y_N$}\label{section:yangian}
When $M=0$ and $\fs=\fs^{st}=0^M1^N$,
the composition $Y_{0|N}(\fs)\stackrel{\zeta_{0|N}}{\rightarrow}Y_{N|0}(\sd) \cong Y_N$ is an algebra isomorphism.
In this section, we will consider the usual Yangian $Y_N$ and describe the center of $Y_N$ in terms of the parabolic generators.

\subsection{Quantum determinants}
Let $X(u)=(x_{i,j}(u))_{i,j=1}^{\ell}$ be an arbitrary matrix whose entries are formal power series in $u^{-1}$ with coefficients from $Y_N$.
Following \cite{MNO96},
one can define the {\it quantum determinant} of the matrix $X(u)$ as follows:
\begin{align}\label{definition of qdet}
\qdet X(u)=\sum\limits_{\sigma\in S_{\ell}}\sgn(\sigma)x_{\sigma(1),1}(u)x_{\sigma(2),2}(u-1)\cdots x_{\sigma(\ell),\ell}(u-\ell+1),
\end{align}
where $S_\ell$ is the symmetric group on $\ell$ elements.

For tuples $\bm{i}=(i_1,\dots,i_d)$ and $\bm{j}=(j_1,\dots,j_d)$ of integers from $\{1,\dots,N\}$,
we let
\begin{align}\label{submatrix yangian}
T_{\bm{i},\bm{j}}(u):=(t_{i_k,j_l}(u))_{1\leq k,l\leq d}
\end{align}
be the corresponding $d\times d$ submatrix of $T(u)$.
By definition, we have
\begin{align}\label{definition of qdet minor}
\qdet T_{\bm{i},\bm{j}}(u)=\sum\limits_{\pi\in S_d}\sgn(\pi)t_{i_{\pi(1),j_1}}(u)t_{i_{\pi(2),j_2}}(u-1)\cdots t_{i_{\pi(d),j_d}}(u-d+1).
\end{align}
We warn the reader that our $T_{\bm{i},\bm{j}}(u)$ is different from the one used in \cite[(8.2)]{BK05}.
The $T_{\bm{i},\bm{j}}(u)$ in {\it loc. cit.} is just our $\qdet T_{\bm{i},\bm{j}}(u)$.

In the special case $\bm{i}=\bm{j}=(1,\dots,N)$,
we denote the quantum determinant $\qdet T_{\bm{i},\bm{j}}(u)$ instead by $c_N(u)$, i.e.
\begin{align}\label{cnu=quantum det}
c_N(u)=\sum\limits_{r\geq 0}c_N^{(r)}u^{-r}:=\qdet T_{(1,\dots,N),(1,\dots,N)}(u)=\qdet T(u).
\end{align}

\subsection{Center of $Y_N$}
According to \cite[Theorem 2.13]{MNO96},
the coefficients $c_N^{(1)},c_N^{(2)},\dots$ of $\qdet T(u)$ are algebraically independent and generate the center $Z(Y_N)$ of $Y_N$.
Moreover, in terms of the Drinfeld generators,
it is well-known (see \cite[Theorem 8.6]{BK05}) that:
\begin{align}\label{cnu=product of di}
c_N(u)=d_1(u)d_2(u-1)\cdots d_N(u-N+1).
\end{align}

For $m\in\ZZ_{\geq 0}$, we write $\psi_m:Y_N\hookrightarrow Y_{m+N}$ for the shift map (see Proposition \ref{morphisms}(4)).
\begin{Lemma}\cite[Lemma 8.5]{BK05}\label{Lemma psi qdet=}
Let $\bm{i},\bm{j}$ be $d$-tuples of distinct integers from $\{1,\dots,N\}$. Then
\[
\psi_m(\qdet T_{\bm{i},\bm{j}}(u))=c_m(u+m)^{-1}\qdet T_{m\#\bm{i},m\#\bm{j}}(u+m),
\]
where $m\#\bm{i}$ denotes the $(m+d)$-tuple $(1,\dots,m,m+i_1,\dots,m+i_d)$ and $m\#\bm{j}$ defined similarly.
\end{Lemma}

Let $\mu=(\mu_1,\dots,\mu_n)$ be a given composition of $N$ with length $n$.
For each $1\leq i\leq n$, we define the partial sum
\begin{align}\label{partial sum}
p_i(\mu):=\mu_1+\cdots+\mu_{i-1}.
\end{align}

The submatrices $\{D_i(u);~1\leq i\leq n\}$ are defined in \eqref{quasideterminant D} (see also \cite[(5.2)]{BK05}).
We would like to generalize (\ref{cnu=quantum det}) and (\ref{cnu=product of di}),
describing $c_N(u)$ in terms of the parabolic generators.
\begin{Proposition}\label{Prop: cnu parabolic}
For $\mu=(\mu_1,\dots,\mu_n)$ and $1\leq i\leq n$, we have
\[
\qdet~^{\mu}D_i(u-p_i(\mu))=\prod\limits_{k=p_i(\mu)+1}^{p_{i+1}(\mu)}d_{k}(u-k+1).
\]
In particular,
\[
c_N(u)=\qdet D_1(u-p_1(\mu))\qdet D_2(u-p_2(\mu))\dots \qdet D_n(u-p_n(\mu)).
\]
\end{Proposition}
\begin{proof}
For $i\geq 1$, we let $\bar{\mu}:=(\mu_i,\mu_{i+1},\dots,\mu_n)$.
By applying the shift map (\ref{psi D1=Da}), we obtain
\[
^{\mu}D_i(u-p_i(\mu))=\psi_{p_i(\mu)}(^{\bar{\mu}}D_{1}(u-p_i(\mu))),
\]
where $^{\bar{\mu}}D_{1}(u-p_i(\mu))\in Y_{\mu_i+\cdots+\mu_{n}}$.
Note that (\ref{quasideterminant D}) implies $^{\bar{\mu}}D_{1}(u)\!=\!^{\bar{\mu}}T_{1,1}(u)$.
Setting $\bm{\mu_i}=(1,\dots,\mu_i)$, we have by Lemma \ref{Lemma psi qdet=} that
\begin{eqnarray*}
\qdet~^{\mu}D_i(u-p_i(\mu))&=&\psi_{h_i}(\qdet~^{\bar{\mu}}T_{1,1}(u-p_i(\mu)))\\
&=&c_{h_i}(u)^{-1}\qdet T_{p_i(\mu)\#\bm{\mu_i},p_i(\mu)\#\bm{\mu_i}}(u)
\end{eqnarray*}
Since the elements $\{d_i(u);~1\leq i\leq N\}$ commute (cf. \cite[(5.9)]{BK05}),
a two-fold application of (\ref{cnu=product of di}) readily yields
\[
\qdet~^{\mu}D_i(u-p_i(\mu))=\prod\limits_{k=p_i(\mu)+1}^{p_{i+1}(\mu)}d_{k}(u-k+1).
\]
The last assertion again  follows from \eqref{cnu=product of di}.
\end{proof}

\section{The center of $\Ymn$}\label{section:superyangian}
\subsection{The super Yangian $\Ymn(\fs^{st})$}\label{section standard 01-seq}
In this subsection, we always assume that $\Ymn=\Ymn(\fs^{st})$ is the super Yangian corresponding to the standard $0^M1^N$-sequence (see \eqref{standard sequence}).
Following \cite{Na91}, we define the {\it quantum Berezinian} ({\it superdeterminant}) of the matrix $T(u)$ as the following power series:
\begin{eqnarray}\label{bmnRTT}
\qsdet T(u)&:=&\sum_{\rho\in S_M}\sgn(\rho)t_{\rho(1),1}(u)t_{\rho(2),2}(u-1)\cdots t_{\rho(M),M}(u-M+1)\nonumber\\
&\times &\sum_{\sigma\in S_N}\sgn(\sigma)t'_{M+1,M+\sigma(1)}(u-M+1)\cdots t'_{M+N,M+\sigma(N)}(u-M+N)\nonumber\\
&=&1+\sum\limits_{r\geq 1}b_{M|N}^{(r)}u^{-1}.
\end{eqnarray}

The elements $\{b_{M|N}^{(r)};~r>0\}$ generate the center $Z(\Ymn)$ of $\Ymn$ (\cite[Theorem 4]{Gow07}).
By \cite[Theorem 1]{Gow05},
the quantum Berezinian can be written in terms of the Drinfeld generators as follows:
\begin{eqnarray}\label{bmndrinfeld}
b_{M|N}(u):=1+\sum\limits_{r\geq 1}b_{M|N}^{(r)}u^{-1}&=&d_1(u)d_2(u-1)\cdots d_M(u-M+1)\nonumber\\
&\times &d_{M+1}(u-M+1)^{-1}\cdots d_{M+N}(u-M+N)^{-1}.
\end{eqnarray}

Now let $\lambda$ be a composition of $M$ with length $m$ and $\nu$ be a composition of $N$ with length $n$.
We let
$\mu=(\lambda_1,\lambda_2,\dots,\lambda_m \mid \nu_1,\nu_2,\dots,\nu_n)$ denote the composition of $(M|N)$,
and we continue to use the notations $p_i(\lambda)$ and $p_i(\nu)$ for the partial sum of $\lambda$ and $\nu$ respectively,
as defined in \eqref{partial sum}.

\begin{Proposition}\label{Prop: bmnu parabolic1}
Associated to $\mu=(\lambda_1,\lambda_2,\dots,\lambda_m \mid \nu_1,\nu_2,\dots,\nu_n)$,
we have
\begin{eqnarray}\label{formula 1}
\qsdet T(u)=\qdet D_1(u-p_1(\lambda))\qdet D_2(u-p_2(\lambda))\dots \qdet D_m(u-p_m(\lambda))\nonumber\\
\times\qdet D_{m+1}'(u-M+p_2(\nu))\qdet D_{m+2}'(u-M+p_3(\nu))\dots \qdet D_{m+n}'(u-M+N).
\end{eqnarray}
\end{Proposition}
\begin{proof}
For convenience, let us write
\[b_1(u):=\qdet D_1(u-p_1(\lambda))\qdet D_2(u-p_2(\lambda))\dots \qdet D_m(u-p_m(\lambda))\]
and
\[b_2(u):=\qdet D_{m+1}'(u-M+p_2(\nu))\qdet D_{m+2}'(u-M+p_3(\nu))\dots \qdet D_{m+n}'(u-M+N),\]
that is $b_1(u)$ (resp. $b_2(u)$) is the first (resp. second) part of the expression in \eqref{formula 1}.

It is clear that $b_1(u)$ is an element of the subalgebra of $\Ymn$ generated by the set $\{t_{i,j}^{(r)};~1\leq i,j\leq M,~r>0\}$.
The subalgebra is isomorphic to the Yangian $Y_M$ by the standard inclusion map $Y_M\hookrightarrow \Ymn$
which sends each generator $t_{i,j}^{(r)}$ to the generator of the same name in $\Ymn$.
Thus, Proposition \ref{Prop: cnu parabolic} in conjunction with \eqref{cnu=product of di}  gives
\[b_1(u)=d_1(u)d_2(u-1)\cdots d_M(u-M+1).\]

Now consider
\[\overleftarrow{\mu}:=(\nu_n,\dots,\nu_1\mid \lambda_m,\dots,\lambda_{1}),\]
the reverse of $\mu$, which is a composition of $(N|M)$.
Owing to \cite[(4.6)]{Peng11},
$b_2(u)$ is the image under the isomorphism $\zeta_{N|M}:Y_{N|M}(\overleftarrow{\mu})\rightarrow\Ymn(\mu)$
of
\[\qdet D_n(u-M+p_2(\nu))\qdet D_{n-1}(u-M+p_3(\nu))\cdots\qdet D_1(u-M+p_{n+1}(\nu)).\]
In view of \cite[(7.7)]{Peng11},
the order of the product here is irrelevant.
By using again Proposition \ref{Prop: cnu parabolic} and \eqref{cnu=product of di}, we obtain
\[b_2(u)=\zeta_{N|M}(c_N(u-M+N))=\zeta_{N|M}(d_1(u-M+N)d_2(u-M+N-1)\cdots d_N(u-M+1)),\]
and \cite[(11)]{Gow07} implies
\[b_2(u)=d_{M+1}(u-M+1)^{-1}\cdots d_{M+N}(u-M+N)^{-1}.\]
Then our assertion follows immediately from \eqref{bmndrinfeld}.
\end{proof}

\subsection{Arbitrary $01$-sequences}
In this section, when necessary, we will add an additional superscsript $\fs$ to our notation to avoid any ambiguity as $\fs$ varies.
We first consider the particular case $M=N=1$.

\begin{Lemma}\label{relations in Y11s}
The following relations hold in $Y_{1|1}(01)[[u^{-1}]]$:
\begin{align}\label{Y11-1}
t_{1,1}(u)t_{2,1}(u-1)=t_{2,1}(u)t_{1,1}(u-1),
\end{align}
\begin{align}\label{Y11-2}
t_{2,2}(u)t_{2,1}(u-1)=t_{2,1}(u)t_{2,2}(u-1),
\end{align}
\begin{align}\label{Y11-3}
t_{2,1}(u)t_{2,1}(u-1)=0,
\end{align}
\begin{align}\label{Y11-4}
t_{1,1}(u)t_{2,2}(u-1)-t_{2,2}(u)t_{1,1}(u-1)=t_{1,2}(u)t_{2,1}(u-1)+t_{2,1}(u)t_{1,2}(u-1).
\end{align}
\end{Lemma}
\begin{proof}
These follow from the defining relation.
Setting $(i,j,k,l)=(1,1,2,1)$ and $v:=u-1$,
we have by \eqref{tiju tklu relation} that
\[[t_{1,1}(u),t_{2,1}(u-1)]=t_{2,1}(u)t_{1,1}(u-1)-t_{2,1}(u-1)t_{1,1}(u).\]
This immediately implies \eqref{Y11-1}.
For \eqref{Y11-2}, we set $(i,j,k,l)=(2,2,2,1)$ and $(i,j,k,l)=(2,1,2,2)$ in \eqref{tiju tklu relation}, respectively.
Then we obtain
\[[t_{2,2}(u),t_{2,1}(u-1)]=t_{2,2}(u-1)t_{2,1}(u)-t_{2,2}(u)t_{2,1}(u-1),\]
and
\[[t_{2,1}(u),t_{2,2}(u-1)]=t_{2,1}(u-1)t_{2,2}(u)-t_{2,1}(u)t_{2,2}(u-1).\]
Thus,
\[2t_{2,2}(u)t_{2,1}(u-1)=t_{2,2}(u-1)t_{2,1}(u)+t_{2,1}(u-1)t_{2,2}(u)=2t_{2,1}(u)t_{2,2}(u-1).\]
This yields \eqref{Y11-2}.
We note the following relation
\[[t_{2,1}(u),t_{2,1}(u-1)]=t_{2,2}(u-1)t_{2,1}(u)-t_{2,2}(u)t_{2,1}(u-1).\]
Since both $t_{2,1}(u)$ and $t_{2,1}(u-1)$ are odd,
the equation \eqref{Y11-3} follows directly from above.
To establish the relation \eqref{Y11-4},
set $(i,j,k,l)=(1,1,2,2)$ and $v:=u-1$ in \eqref{tiju tklu relation}, simplify, we have
\begin{align}\label{Y11-5}
t_{1,1}(u)t_{2,2}(u-1)-t_{2,2}(u-1)t_{1,1}(u)=t_{2,1}(u)t_{1,2}(u-1)-t_{2,1}(u-1)t_{1,2}(u).
\end{align}
One obtains similarly
\begin{align}\label{Y11-6}
t_{2,2}(u-1)t_{1,1}(u)-t_{2,2}(u)t_{1,1}(u-1)=t_{1,2}(u)t_{2,1}(u-1)+t_{2,1}(u-1)t_{1,2}(u).
\end{align}
The relation \eqref{Y11-4} now follows from \eqref{Y11-5} and \eqref{Y11-6}.
\end{proof}
As observed in \cite{Peng14}, the definition of $Y_{M|N}$ is independent of the choices of $\fs$.
It is clear that any re-labeling map $t_{i,j}(u)\mapsto t_{\sigma(i),\sigma(j)}(u)$,
preserving the partities of the generators, where $\sigma$ is a permutation in the symmetric group $S_{M+N}$,
extends to an isomorphism between the respective super Yangians.

In particular, the super Yangian $Y_{1|1}(01)$ and $Y_{1|1}(10)$ are isomorphic via the assignment
$t_{i,j}^{10}(u)\mapsto t_{3-i,3-j}^{01}(u)$.
Denote by $\sigma_{1}$ the corresponding isomorphism.

The following two results have appeared in \cite[Section 3.3]{Lu22} (see also \cite[Remark 3.11]{Lu21}, \cite[Proposition 3.6]{HM20} and \cite[Section 4.3]{LM21}).
Here we provide a detailed proof.

\begin{Lemma}\label{Y11d-iso}
Let $\sigma_{1}: Y_{1|1}(10)\rightarrow Y_{1|1}(01)$ be the algebra isomorphism defined above.
Then
\[\sigma_{1}(d_1^{10}(u)^{-1}d_2^{10}(u))=d_1^{01}(u-1)d_2^{01}(u-1)^{-1}.\]
\end{Lemma}
\begin{proof}
The definition of quasideterminants \eqref{quasideterminant D} implies that
\[d_1^{10}(u)^{-1}d_2^{10}(u)=t_{1,1}^{10}(u)^{-1}(t_{2,2}^{10}(u)-t_{2,1}^{10}(u)t_{1,1}^{10}(u)^{-1}t_{1,2}^{10}(u)),\]
so that
\[\sigma_{1}(d_1^{10}(u)^{-1}d_2^{10}(u))=t_{2,2}^{01}(u)^{-1}(t_{1,1}^{01}(u)-t_{1,2}^{01}(u)t_{2,2}^{01}(u)^{-1}t_{2,1}^{01}(u)).\]
Then it is equivalent to show that
\begin{eqnarray*}
t_{2,2}^{01}(u)t_{1,1}^{01}(u-1)&=&(t_{1,1}^{01}(u)-t_{1,2}^{01}(u)t_{2,2}^{01}(u)^{-1}t_{2,1}^{01}(u))\\
&&\times(t_{2,2}^{01}(u-1)-t_{2,1}^{01}(u-1)t_{1,1}^{01}(u-1)^{-1}t_{1,2}^{01}(u-1)).
\end{eqnarray*}
This follows from \eqref{Y11-1}-\eqref{Y11-4}.
\end{proof}

More generally, let $\fs=(s_1\cdots s_{M+N})$ be a $01$-sequence.
Assume that $\fs\neq \fs^{st}$,
i.e., $\fs$ is not the standard $01$-sequence defined in \eqref{standard sequence}.
Then there exists $1\leq i\leq M+N-1$ such that $(s_is_{i+1})=(10)$.
Fix $i$ with $(s_is_{i+1})=(10)$ and let $\bar\fs$ be the $01$-sequence obtained from $\fs$ by switching $s_i$ and $s_{i+1}$.
We have the following isomorphism:
\begin{align}\label{sigmai}
\sigma_i: \Ymn(\fs)\rightarrow\Ymn(\bar\fs);~t_{k,l}^{\fs}(u)\mapsto t_{\sigma_i(k),\sigma_i(l)}^{\bar\fs}(u),
\end{align}
where  $\sigma_i$ is the simple reflection $(i,i+1)$ in the symmetric group $S_{M+N}$ and we still denote the induced isomorphism by $\sigma_i$.

\begin{Remark}
As pointed out to us by the referee, the above isomorphism \eqref{sigmai} is exactly the odd reflection considered in
\cite{Mol22} and \cite{Lu22},
in those papers they are focused on the representation theory of $\Ymn(\fs)$.
\end{Remark}

\begin{Proposition}\label{YMNd-iso}
We have
\[\sigma_{i}(d_i^{\fs}(u)^{-1}d_{i+1}^{\fs}(u))=d_i^{\bar\fs}(u-1)d_{i+1}^{\bar\fs}(u-1)^{-1}.\]
\end{Proposition}
\begin{proof}
If $i=1$, then $(s_1s_2)=(10)$.
There is a standard embedding $Y_{1|1}(10)\hookrightarrow Y_{M|N}(\fs)$ under which $t_{i,j}(u)\in Y_{1|1}(10)$ maps to the element with the same name in $Y_{M|N}(\fs)$.
In this case, Lemma \ref{Y11d-iso} yields the assertion.

Assume that $i>1$. We put $\fs_1:=(s_1\cdots s_{i-1})$ and $\fs_2:=(s_i\cdots s_{M+N})$,
so that $\fs=\fs_1\fs_2$.
Moreover, we let $\bar\fs_2$ be the $01$-sequence obtained from $\fs_2$ by switching $s_i$ and $s_{i+1}$.
Clearly, $\bar\fs=\fs_1\bar\fs_2$.
A two-fold application of the shift map (see \eqref{psi D1=Da}) in conjunction with \eqref{psit} yields
\begin{align*}
\sigma_{i}(d_i^{\fs}(u)^{-1}d_{i+1}^{\fs}(u))&=\sigma_i(\psi_{i-1}(d_1^{\fs_2}(u)^{-1}d_{2}^{\fs_2}(u)))\\
&=\psi_{i-1}(\sigma_1(d_1^{\fs_2}(u)^{-1}d_{2}^{\fs_2}(u)))\\
&=\psi_{i-1}(d_1^{\bar\fs_2}(u-1)^{-1}d_{2}^{\bar \fs_2}(u-1))\\
&=d_i^{\bar\fs}(u-1)^{-1}d_{i+1}^{\bar \fs}(u-1).
\end{align*}
\end{proof}

Given  a $01$-sequence $\fs=(s_1\cdots s_{M+N})$,
there exists $1\leq i_1<i_2<\cdots<i_M\leq M+N$ and $1\leq j_1<j_2<\cdots<j_N\leq M+N$
such that $s_{i_k}=0$ and $s_{j_l}=1$ for all $1\leq k\leq M$ and $1\leq l\leq N$.
Obviously, the $i's$ and $j's$ are uniquely determined by $\fs$.
\begin{Definition}
{\it The quantum Berezinian (superdeterminant) of the matrix $T^{\fs} (u)$ is defined via}
\begin{eqnarray}\label{qsdet-arbitrary}
\qsdet T^{\fs}(u):=&\sum\limits_{\rho\in S_M}\sgn(\rho)t_{i_{\rho(1)},i_1}(u)t_{i_{\rho(2)},i_2}(u-1)\cdots t_{i_{\rho(M)},i_M}(u-M+1)\nonumber\\
&\times\sum\limits_{\sigma\in S_N}\sgn(\sigma)t'_{j_1,j_{\sigma(1)}}(u-M+1)\cdots t'_{j_N,j_{\sigma(N)}}(u-M+N).
\end{eqnarray}
\end{Definition}

\begin{Remark}
(1) When $\fs=\fs^{st}$ is the standard $01$-sequence,
we have $i_k=k$ and $j_l=M+l$ for all $1\leq k\leq M$ and $1\leq l\leq N$,
so that \eqref{qsdet-arbitrary} recovers Nazarov's original definition \eqref{bmnRTT}.

(2) For any permutations $\theta \in S_M$ and $\tau \in S_N$, we have by \cite[Remark 2.8]{MNO96} that
\begin{eqnarray}\label{permutation qsdet}
\qsdet T^{\fs}(u)=&\sgn(\theta)\sum\limits_{\rho\in S_M}\sgn(\rho)t_{i_{\rho(1)},i_{\theta(1)}}(u)t_{i_{\rho(2)},i_{\theta(2)}}(u-1)\cdots t_{i_{\rho(M)},i_{\theta(M)}}(u-M+1)\nonumber\\
&\times\sgn(\tau)\!\!\!\sum\limits_{\sigma\in S_N}\sgn(\sigma)t'_{j_{\tau(1)},j_{\sigma(1)}}(u-M+1)\cdots t'_{j_{\tau(N)},j_{\sigma(N)}}(u-M+N).
\end{eqnarray}

\end{Remark}

Observe that the permutation $\sigma_\fs\in S_{M+N}$ defined by $i_k\mapsto k, j_l\mapsto M+l;~1\leq k\leq M,~1\leq l\leq N$ provides
an isomorphism (which is still denoted by $\sigma_\fs$) :
\begin{align}\label{sigmas}
\sigma_\fs:\Ymn(\fs)\rightarrow\Ymn(\fs^{st});~t_{i,j}^{\fs}(u)\mapsto t_{\sigma_{\fs}(i),\sigma_{\fs}(j)}^{\fs^{st}}(u).
\end{align}
In particular,
\begin{align}\label{sigmasqsdet=qsdet}
\sigma_\fs(\qsdet T^\fs(u))=\qsdet T^{\fs^{st}}(u).
\end{align}
This in combination with \cite[Theorem 4]{Gow07} (see Section \ref{section standard 01-seq}) yields that
the coefficients of the quantum Berezinian $\qsdet T^\fs(u)$ generate the center $Z(\Ymn(\fs))$~of $\Ymn(\fs)$.

Following \cite{Tsy20}, define
\begin{align}\label{expression of bmns}
b_{M|N}^{\fs}(u):=1+\sum\limits_{r\geq 1}b^{\fs(r)}u^{-r}=\tilde{d}_1(u_1)\tilde{d}_2(u_2)\cdots\tilde{d}_{M+N}(u_{M+N}),
\end{align}
where
\begin{align}\label{definition:tilde d}
\tilde{d}_i(u):=
                \left\{
                        \begin{aligned}
                        &d_i(u)\ \ \ \ \text{if}~s_i=0\\
                        &d_i(u)^{-1}~\text{if}~s_i=1
                        \end{aligned}
                \right.,
 \end{align}
while
\begin{align}\label{definition:ui}
u_1=u+s_1
~\text{and}~u_{i+1}=
              \left\{
                        \begin{aligned}
                        &u_i-1\ \ \ \ \text{if}~s_i=s_{i+1}=0\\
                         &u_i+1\ \ \ \ \text{if}~s_i=s_{i+1}=1\\
                         &u_i\ \ \ \ \ \ \ \ \ \ \ \text{if}~s_i\neq s_{i+1}\\
                        \end{aligned}
                \right..
  \end{align}
Note that our definition of $u_1$ is slightly different from that given in \cite[(2.41)]{Tsy20}.

The following result was obtained by Tsymbaliuk \cite[Theorem 2.43]{Tsy20},
who used the defining relations of the Drinfeld-type presentation of $\Ymn(\fs)$ \cite{Gow07}, \cite{Peng16}.
\begin{Theorem}\label{Theorem:bmn=qsdet}
We have $b_{M|N}^{\fs}(u)=\qsdet T^{\fs}(u)$.
In particular, the elements $\{b^{\fs(r)};~r\geq 1\}$ are central.
\end{Theorem}
\begin{proof}
If $\fs$ is the standard $01$-sequence,
then $\qsdet T^{\fs}(u)$ is just the usual quantum Berezinian \eqref{bmnRTT},
and \eqref{bmndrinfeld} yields the assertion.

Alternatively, there exists an integer $1\leq i\leq M+N-1$ such that $(s_is_{i+1})=(10)$.
We put
\[k:=\min\{i;~(s_{i}s_{i+1})=(10),~1\leq i\leq M+N-1\}.\]
Let $\sigma_{k}$ be the isomorphism induced by the simple reflection $(k,k+1)\in S_{M+N}$ (see \ref{sigmai}).
We denote by $\bar\fs:=(\bar s_{1}\cdots\bar s_{M+N})$ the $01$-sequence obtained from $\fs$ by switching $s_{k}$ and $s_{k+1}$.
Note that $\tilde{d}_k(u_k)\tilde{d}_{k+1}(u_{k+1})=d_{k}(u_{k})^{-1}d_{k+1}(u_{k})$ in the expression \eqref{expression of bmns}.
Proposition \ref{YMNd-iso} in conjunction with \cite[I.2.5]{GR97} implies that
\[
\sigma_{k}(b_{M|N}^{\fs}(u))=b_{M|N}^{\bar\fs}(u).\]
Repeat the above argument with $\fs$ replaced by $\bar\fs$ until we arrive at the standard $01$-sequence.
Note that the composition of all those simple transpositions $\sigma_k$ applied in the process equals to the permutation $\sigma_\fs$.
Then we obtain
\[\sigma_{\fs}(b_{M|N}^{\fs}(u))=b_{M|N}^{\fs^{st}}(u)\in\Ymn(\fs^{st})[[u^{-1}]].\]
The assertion thus follows from \eqref{bmndrinfeld} and \eqref{sigmasqsdet=qsdet}.
\end{proof}

\begin{Example}
Let $\fs=(1010)$ be a $0^{2}1^{2}$-sequence.
By definition, $(i_{1},i_{2})=(2,4)$, $(j_{1},j_{2})=(1,3)$.
Therefore
\begin{eqnarray*}
\qsdet T^{\fs}(u)=&(t_{2,2}(u)t_{4,4}(u-1)-t_{4,2}(u)t_{2,4}(u-1))\\
&\times(t'_{1,1}(u-1)t'_{3,3}(u)-t'_{1,3}(u-1)t'_{3,1}(u)).
\end{eqnarray*}
The permutation $\sigma_{\fs}\in S_{4}$ is given by
$2\mapsto 1, 4\mapsto 2, 1\mapsto 3, 3\mapsto 4$.
The map $\sigma_{\fs}$ \eqref{sigmas} can be written as the composition of the following isomorphisms:
\[\sigma_{\fs}:Y_{2|2}(1010)\stackrel{\sigma_{1}}{\longrightarrow}Y_{2|2}(0110)\stackrel{\sigma_{3}}{\longrightarrow}Y_{2|2}(0101)\stackrel{\sigma_{2}}{\longrightarrow}Y_{2|2}(0011);\]
\[b_{2|2}^{1010}(u)\mapsto b_{2|2}^{0110}(u)\mapsto b_{2|2}^{0101}(u)\mapsto b_{2|2}^{0011}(u).\]
\end{Example}

More generally, for each $1\leq i\leq M+N$, we let
$$T_{(1,\dots,i),(1,\dots i)}(u):=(t_{k,l}(u))_{1\leq k,l\leq i}$$
be the $i\times i$ submatrix of $T^{\fs}(u)$ whose both rows and columns are enumerated by the set $\{1,\dots,i\}$ , and set
\[\fs_{i,1}:=(s_1,\dots,s_i)\quad\quad \fs_{i,2}:=(s_{i+1},\dots,s_{M+N}).\]
It follows that $\fs=\fs_{i,1}\fs_{i,2}$. The {\it quantum superminor} $\qsdet T_{(1,\dots,i),(1,\dots,i)}^{\fs_{i,1}}$ can be defined in the same way as \eqref{qsdet-arbitrary}.
The following corollary is a superalgebra generalization of \cite[Theorem 8.7 (i)]{BK05}.
\begin{Corollary}
For $1\leq i\leq M+N$, we have
\[\tilde{d}_i(u_i)=\qsdet T_{(1,\dots,i-1),(1,\dots,i-1)}^{\fs_{i-1,1}}(u)^{-1}\qsdet T_{(1,\dots,i),(1,\dots,i)}^{\fs_{i,1}}(u)\]
\end{Corollary}
\begin{proof}
Denote by $m_i$ and $n_i$ the number of $0$'s and $1$'s in $\fs_{i,1}$ repectively,
so that $\fs_{i,1}$ is a $0^{m_i}1^{n_i}$-sequence.
It is easy to see that the subalgebra which is generated by the elements $\{t_{k,l}^{(r)}\}_{k,l=1}^i$ is isomorphic to the super Yangian $Y_{m_i|n_i}(\fs_{i,1})$.
We consider the standard embedding $\iota: Y_{m_i|n_i}(\fs_{i,1})\hookrightarrow \Ymn(\fs)$ under which
$t_{i,j}^{(r)}\in Y_{m_i|n_i}(\fs_{i,1})$ maps to the element with the same name in $ \Ymn(\fs)$.
Theorem \ref{Theorem:bmn=qsdet} yields
\[\qsdet T_{(1,\dots,i),(1,\dots,i)}^{\fs_{i,1}}(u)=b_{m_i|n_i}^{\fs_{i,1}}(u)\in Y_{m_i|n_i}(\fs_{i,1})[[u^{-1}]].\]
Owing to \eqref{quasideterminant D}, we obtain
\[\qsdet T_{(1,\dots,i),(1,\dots,i)}^{\fs_{i,1}}(u)=\tilde{d}_1(u_1)\tilde{d}_2(u_2)\cdots\tilde{d}_{i}(u_{i})\in \Ymn(\fs)[[u^{-1}]].\]
Then direct computation gives the assertion.
\end{proof}

Now we consider the parabolic generators defined in Section \ref{subsection:parabolic generators}.
Fix a $0^{M}1^{N}$-sequence $\fs=(s_{1}\dots s_{M+N})$ and let $\mu=(\mu_1,\ldots,\mu_n)$ be a given composition of $M+N$ with length $n$.
We break $\fs=\fs_1\fs_2\cdots\fs_n$ into $n$ subsequences according to $\mu$,
where $\fs_1$ is the subsequence consisting of the first $\mu_1$ digits of $\fs$, $\fs_2$ is the subsequence consisting of the next $\mu_2$ digits of $\fs$, and so on.

For each $1\leq a\leq n$, let $p_a$ and $q_a$ denote the number of $0$'s and $1$'s in $\fs_a$, respectively.
Then each $\fs_a$ is a $0^{p_a}1^{q_a}$-sequence of $\gl_{p_a|q_a}$,
and for all $1\leq i\leq\mu_{a}$,
put
\[a\star i:=\mu_{1}+\cdots+\mu_{a-1}+i=p_a(\mu)+i\] and
define the {\em restricted parity} $|i|_a$ by
\begin{equation*}
|i|_a:=\mbox{ the~} i-\mbox{th~ digits~ of~} \fs_a=|a\star i|.
\end{equation*}
Moreover, there exists $1\leq i_{1}<i_{2}<\cdots<i_{p_{a}}\leq \mu_{a}$ and $1\leq j_{1}<j_{2}<\cdots<j_{q_{a}}\leq \mu_{a}$
such that $|i_{k}|_{a}=0$ and  $|j_{l}|_{a}=1$ for all $1\leq k\leq p_{a}$ and $1\leq l\leq q_{a}$.
Similar to \eqref{qsdet-arbitrary} we define the quantum Berezinian of the diagonal element $D_{a}(u)$ \eqref{definition of Da} by
\begin{eqnarray}\label{qsdet-Da}
\qsdet D_{a}(u):=&\sum\limits_{\rho\in S_{p_{a}}}\sgn(\rho)D_{a;i_{\rho(1)},i_{1}}(u)D_{a;i_{\rho(2)},i_{2}}(u-1)\cdots D_{a;i_{\rho(p_{a})},i_{p_{a}}}(u-p_{a}+1)\nonumber\\
\times &\sum\limits_{\sigma\in S_{q_{a}}}\sgn(\sigma)D'_{a;j_{1},j_{\sigma(1)}}(u-p_{a}+1)\cdots D'_{a;j_{q_{a}},j_{\sigma(q_{a})}}(u-p_{a}+q_{a}).
\end{eqnarray}

Let $\{u_{i};~1\leq i\leq M+N\}$ be defined by \eqref{definition:ui}.
In particular, $u_{a\star 1}=u_{p_a(\mu)+1}$.

\begin{Theorem}\label{Theorem:qsdet=prod qsdet}
let $\mu=(\mu_1,\ldots,\mu_n)$ be a given composition of $M+N$ and $\fs=(s_{1}\cdots s_{M+N})$ be a $0^{M}1^{N}$-sequence.
For $a\geq 1$, we have
\[
\qsdet~^{\mu}D_{a}(u_{a\star1}-|1|_{a})=\tilde{d}_{a\star 1}(u_{a\star 1})\tilde{d}_{a\star 2}(u_{a\star 2})\cdots \tilde{d}_{a\star \mu_{a}}(u_{a\star\mu_{a}})
\]
In particular,
\[\qsdet T^{\fs}(u)=\qsdet D_1(u)\qsdet D_{2}(u_{2\star1}-|1|_{2})\cdots \qsdet D_{n}(u_{n\star1}-|1|_{n}).\]

\end{Theorem}
\begin{proof}
For $1\leq a\leq n$,
we let $\bar{\mu}:=(\mu_a,\mu_{a+1},\dots,\mu_n)$ and $\bar\fs:=\fs_{a}\cdots\fs_{n}$.
It follows from \eqref{psi D1=Da} that
\[
^{\mu}D_{a}(u_{a\star1}-|1|_{a})=\psi_{p_a(\mu)}(^{\bar\mu}D_{1}(u_{a\star1}-|1|_{a})),
\]
where $^{\bar\mu}D_{1}(u_{a\star1}-|1|_{a})\in Y_{\bar\mu}(\bar\fs)$.
In view of \eqref{quasideterminant D}, we have $^{\bar{\mu}}D_{1}(u)\!=\!^{\bar{\mu}}T_{1,1}(u)$,
so that Theorem \ref{Theorem:bmn=qsdet} ensures that
\begin{eqnarray*}
\qsdet~^{\mu}D_{a}(u_{a\star1}-|1|_{a})&=&\psi_{p_a(\mu)}(\qsdet~^{\bar{\mu}}T_{1,1}(u_{a\star1}-|1|_{a}))\\
&=&\psi_{p_a(\mu)}(b_{p_{a}|q_{a}}^{\fs_{a}}(u_{a\star 1}-|1|_{a}))\\
&=&\tilde{d}_{a\star 1}(u_{a\star 1})\tilde{d}_{a\star 2}(u_{a\star 2})\cdots \tilde{d}_{a\star \mu_{a}}(u_{a\star\mu_{a}}),
\end{eqnarray*}
here the last equality again follows from \eqref{psi D1=Da}.
Since the restricted parity $|1|_{1}$ is just $s_{1}$,
the assertion follows from Theorem \ref{Theorem:bmn=qsdet}  and the definition of $b_{M|N}^{\fs}(u)$ \eqref{expression of bmns}.
\end{proof}

\begin{Example}
Let $\fs=(1010)$ be a $0^{2}1^{2}$-sequence and $\mu=(2,2)$,
so that
\[u_1=u+1, u_2=u+1, u_3=u_{2\star 1}=u+1, u_4=u_{2\star 2}=u+1~\text{and}~|1|_1=s_1=1,~|1|_2=s_3=1.\]
By definition \eqref{qsdet-Da}, we have
\[\qsdet D_1(u)=D_{1;2,2}(u)D'_{1;1,1}(u)\]
and
\[\qsdet D_{2}(u_{2\star1}-|1|_{2})=\qsdet D_{2}(u)=D_{2;2,2}(u)D'_{2;1,1}(u).\]
Consequently,
\[\qsdet T^{1010}(u)=D_{1;2,2}(u)D'_{1;1,1}(u)D_{2;2,2}(u)D'_{2;1,1}(u).\]
\end{Example}

\bigskip

\bigskip

\begin{center}
\textbf{Acknowledgement}
\end{center}
It is a pleasure to thank Kang Lu for helpful comments.
We are particularly grateful to the referee for a very detailed reading of this paper and many constructive
comments and suggestions which improve the original manuscript.
This work is supported by the National Natural Science Foundation of China (Grant Nos. 11801204, 11801394) and the Fundamental Research Funds for the
Central Universities (No. CCNU22QN002).

\end{document}